\documentclass[12pt]{amsart}
\usepackage[utf8]{inputenc}
\usepackage[margin=1in]{geometry}
\usepackage{float,graphicx,subcaption}
\usepackage{enumitem,nameref}
\usepackage{hyperref}
\usepackage{gokulsty}

\allowdisplaybreaks

\makeatletter
\let\orgdescriptionlabel\descriptionlabel
\renewcommand*{\descriptionlabel}[1]{%
  \let\orglabel\label
  \let\label\@gobble
  \phantomsection
  \edef\@currentlabel{#1\unskip}%
  \let\label\orglabel
  \orgdescriptionlabel{#1}%
}
\makeatother

\title[Nonlinearly Elastic Maps]{Nonlinearly Elastic Maps: Energy Minimizing Configurations of Membranes on Prescribed Surfaces}

\author{Timothy J.~Healey}
\address{Department of Mathematics, Cornell University}
\email{tjh10@cornell.edu}

\author{Gokul G.~Nair}
\address{Center for Applied Mathematics, Cornell University}
\email{gn234@cornell.edu}



\begin{document}

\begin{abstract}
We propose a model for nonlinearly elastic membranes undergoing finite deformations while confined to a regular frictionless surface in $\RR^3$. This is a physically correct model of the analogy sometimes given to motivate harmonic maps between manifolds. The proposed energy density function is convex in the strain pair comprising the deformation gradient and the local area ratio. If the target surface is a plane, the problem reduces to 2-dimensional, polyconvex nonlinear elasticity addressed by J.M. Ball. On the other hand, the energy density is not rank-one convex for unconstrained deformations into $\RR^3$. We show that the problem admits an energy-minimizing configuration when constrained to lie on the given surface. For a class of Dirichlet problems, we demonstrate that the minimizing deformation is a homeomorphism onto its image on the given surface and establish the weak Eulerian form of the equilibrium equations.  
\end{abstract}

\maketitle

\section{Introduction}\label{sec:intro}
A harmonic map between two manifolds is typically defined as a critical point of an associated Dirichlet energy. An analogy is sometimes made with the placement of an elastic membrane onto a rigid surface~\cite{eells1978report}. Existence of solutions via the direct method of the calculus of variations is a common approach, e.g.,~\cite{lin2008analysis}. While the Dirichlet energy is clearly not appropriate as a model for nonlinear elasticity (e.g.,~\cite{antman2005problems,ciarlet1988mathematical}), we explore that overall strategy here in a physically correct setting. We postulate a model for a two-dimensional nonlinearly elastic membrane undergoing finite deformations constrained to lie on a smooth rigid surface in $\RR^3$.  
     
Mathematical models for bulk nonlinear elasticity, i.e., involving finite deformations of domains in $\RR^n$ into $\RR^n$, are well known, e.g.,~\cite{antman2005problems,ball1976convexity,ball2002some,ciarlet1988mathematical}. The same cannot be said for two-dimensional membranes deforming in $\RR^3$. Here we propose  
an energy density that is convex in the pair comprising the deformation gradient and the local area ratio. For planar deformations in $\RR^2$, this is the same as two-dimensional polyconvex elasticity as introduced by Ball~\cite{ball1976convexity}; the existence of an energy minimizing configuration in bulk nonlinear elasticity was first proven in that celebrated work. However, for unconstrained mappings into $\RR^3$, membrane energy functions that include ours are not rank-one convex, cf. Appendix. Thus, our proposed energy density is clearly not polyconvex, which would imply rank-one convexity, cf.~\cite{dacorogna2007direct}; see also Remark 1 in Section 2. In any case, energy minimization for the unconstrained membrane problem is not well posed.

\sloppy
In the context of wrinkling models for finitely stretched sheets, our proposed constitutive hypothesis plays an important role in establishing the existence of an energy minimizing configuration~\cite{healey2023existence}. A specific version featuring the same properties has also been employed successfully as part of the model in a numerical-bifurcation setting~\cite{li2016stability}. More generally, it is an important ingredient in models for nonlinearly elastic shells undergoing finite deformations~\cite{healey2023energy}. Each of these also incorporate bending energy, which compensates for the lack of convexity in the membrane part. The prescribed surface in this work plays a similar role, and we prove the existence of an energy-minimizing configuration. We view this as another confirmation of our constitutive hypothesis.

\fussy
We mention that rigorous derivations of two-dimensional membrane energies via dimension reduction from three-dimensional nonlinear elasticity are well known~\cite{anza2008nonlinear,conti2006derivation,le1995nonlinear}. Obtained by $\Gamma$-convergence, the resulting energy functional yields a well-posed minimization problem for an unconstrained membrane in $\RR^3$. In particular, the associated stored energy density is quasiconvex. However, any such model necessarily exemplifies tension-field theory, i.e., only non-compressive configurations are attainable, cf.~\cite{pipkin1994relaxed}. In particular, this rules out the analysis of wrinkling patterns, which arise from competition between nonzero compression and bending, e.g.,~\cite{li2016stability}.
     
The outline of the work is as follows. We formulate the surface-constrained problem in Section~\ref{sec:problem-formulation} and introduce the basic hypotheses for the membrane model. Aside from the features already discussed above, we assume that the energy density function grows unbounded as the local area ratio approaches zero. As in bulk nonlinear elasticity, this condition prevents local interpenetration of matter. The $C^1$ target surface is presumed oriented, and a crucial ingredient of our analysis is the ability to express the oriented local area ratio employing the unit normal field on the prescribed surface. As a consequence, the weak convergence of the oriented local area ratio is obtained in Section~\ref{sec:existence}, and we are able to prove the existence of an energy-minimizing configuration.  As is the case in bulk nonlinear elasticity~\cite{ball2002some}, it is unclear how to take a rigorous first variation at a minimizer, which would lead to the weak form of the Euler-Lagrange equations. This is a consequence of the above-mentioned blow-up of the energy density as the local area ratio approaches zero.
 
In Section~\ref{sec:global-invertibility}, we specialize to a class of Dirichlet problems. Inspired by the methods of~\cite{ball1981global,ball2002some} for bulk elasticity, we first show that an energy-minimizing deformation is a homeomorphism onto its image. With the benefit of an additional, physically reasonable constitutive hypothesis similar to that proposed in~\cite{ball2002some}, we then demonstrate in Section~\ref{sec:weak-form} that 
the weak form of the spatial (Eulerian) equilibrium equations are satisfied. The geometry of the target surface must be accounted for throughout our analysis, and we adapt the appropriate theorems from~\cite{ball1981global,ball2002some} accordingly. In particular, our arguments are facilitated by the use of mixed local coordinates, common in inverse and semi-inverse problems of 
nonlinear elasticity, e.g.,~\cite{doyle1956nonlinear}. Naturally, our problem falls into the latter category. We make some final remarks in Section~\ref{sec:concluding-remarks}. Among other things, we verify our convexity hypothesis for a class of isotropic membrane energies. 

\section{Problem Formulation}\label{sec:problem-formulation}
Let $\Omega\subset\RR^2$ be an open bounded domain with a locally Lipschitz boundary~\cite{adams2003sobolev} $\partial\Omega$, and let $\mathcal N\subset\RR^3$ be a regular, oriented surface~\cite{do2016differential} without boundary. Let $\bm n:\mathcal N\rightarrow \SS^2$ denote a continuous unit-normal field. A configuration is specified by a mapping $\bm f:\bar\Omega\rightarrow\mathcal N$. We denote the gradient of $\bm f$ at $x\in\Omega$ by $\bm F(x):=\nabla\bm f(x)$. We require the local orientation condition
\begin{align}\label{eqn:local-orientation}
    J[\bm f]:=\bm n(\bm f)\cdot(\bm f_{,1}\times\bm f_{,2})=\abs{\bm f_{,1}\times\bm f_{,2}}>0\text{ in }\bar\Omega,
\end{align}
where $\bm f_{,\alpha}$, $\alpha=1,2$ denote partial derivatives and $\bm a\times \bm b$ is the right-handed cross product in $\RR^3$.

We assume the existence of a stored-energy function, $W:\bar\Omega\times\RR^{3\times 2}_+\rightarrow[0,\infty)$, where
\begin{align*}
    \RR^{3\times2}_+:=\left\{\bm F\in\RR^{3\times2}:\det(\bm F^T\bm F)>0\right\}.
\end{align*}
$W(x,\cdot)$ is also assumed to satisfy material objectivity, viz., 
\begin{align}\label{eq:material-objectivity}
    W(x,\bm Q\bm F)=W(x,\bm F) \text{ for all }\bm Q\in SO(3).
\end{align}

We further require $W$ to satisfy the following hypotheses:
\smallskip

\begin{description}
    \item[(H1)\label{itm:growth1}] For $p>4/3$, $q>1$ there exist constants $C_1>0$ and $C_2$ such that
    \begin{equation*}
        W(x,\bm F)\geq C_1\left\{\abs{\bm F}^p +\abs{J}^q\right\} +C_2.
    \end{equation*}
    
    \item[(H2)\label{itm:polyconvexity}] There is a function $\Phi:\bar\Omega\times\RR^{3\times 2}\times(0,\infty)\rightarrow[0,\infty)$, such that $\Phi(\cdot,\bm F,J)$ is measurable for all $(\bm F,J)\in\RR^{3\times 2}\times(0,\infty)$,
    \begin{align*}
        (\bm F,J)\mapsto \Phi(x,\bm F,J)\text{ is convex for almost all }x\in\Omega,
    \end{align*}
    and $W(x,\bm F)\equiv\Phi(x,\bm F,J[\bm f])$.
    
    \item[(H3)\label{itm:growth2}] $\Phi\rightarrow +\infty$ as $J\rightarrow0^+.$
\end{description}

\begin{remark}
    Hypothesis~\ref{itm:polyconvexity} is not the same as polyconvexity. The latter entails a convex function of the four arguments $(\bm F,m_1,m_2,m_2)$, where $\bm F\in \RR^{3\times2}$ and $m_j$, $j=1,2,3$ denote the three independent $2\times2$ sub-determinants of $\bm F$ (not all zero),  cf.~\cite[p.~157]{dacorogna2007direct},~\cite{healey2023energy}.
\end{remark}

Let $L^p(\Omega,\RR^3)$ denote the usual space of $L^p$-integrable 3-vector valued functions on $\Omega$ and let $W^{1,p}(\Omega,\RR^3)$ denote the Sobolev space of vector fields whose weak partial derivatives are also $L^p$-integrable. The norms on these spaces are defined by
\begin{align*}
    &\norm{\bm f}^p_{L^p(\Omega,\RR^3)}=\int_\Omega\abs{\bm f}^p\dif x,\\
    &\norm{\bm f}^p_{W^{1,p}(\Omega,\RR^3)}=\norm{\bm f}^p_{L^p(\Omega,\RR^3)}+\int_\Omega\abs{\nabla\bm f}^p\dif x.
\end{align*}

The Sobolev space of maps from $\Omega$ to $\mathcal N$ is defined by
\begin{align*}
    W^{1,p}(\Omega,\mathcal N)=\left\{\bm u\in W^{1,p}(\Omega,\RR^3):\bm u(x)\in\mathcal N\text{ for a.e.~}x\in\Omega\right\}.
\end{align*}
Note that $W^{1,p}(\Omega,\mathcal N)$ is a weakly closed subset of $W^{1,p}(\Omega,\RR^3)$. Let $\Gamma\subset\partial\Omega$ with 1-dimensional Hausdorff measure, $\abs{\Gamma}_{\partial\Omega}>0$. For a given $\bm f_o\in W^{1,p}(\Omega,\mathcal N)$, we define the admissible set
\begin{align}
\label{eq:admissible-set}
    \mathcal{A}:=\left\{\bm f\in W^{1,p}(\Omega,\mathcal N):J\in L^q(\Omega);\, J>0\text{ a.e.~in }\Omega;\, \bm f|_{\Gamma}=\bm f_o|_{\Gamma}\right\},
\end{align}
where the boundary prescription is understood in the trace sense and $J:=J[\bm f]$.

The total energy of a configuration is given by
\begin{align}\label{eqn:energy}
    E[\bm f] = \int_\Omega W(x,\nabla\bm f(x))\dif x. 
\end{align}

\section{Existence of Energy Minimizers}\label{sec:existence}
Our main existence result is the following:
\begin{theorem}\label{thm:existence}
    Suppose that $\mathcal A$ is non-empty with $\inf_{\mathcal A}E[\bm f]<\infty$. Then there exists $\bm f^*\in\mathcal A$ such that $E[\bm f^*]=\inf_{\mathcal A}E[\bm f]$.
\end{theorem}
\begin{proof}
    Integrating the growth condition~\ref{itm:growth1} yields
    \begin{align*}
        \int_\Omega W(x,\nabla\bm f)\dif x\geq C_1\left\{\norm{\nabla\bm f}^p_{L^p}+\norm{J}^q_{L^q}\right\}+C_2'.
    \end{align*}
    Due to the boundary condition specified in~\eqref{eq:admissible-set}, we may use a Poincar\'e inequality to obtain the coerciveness condition,
    \begin{align}\label{eqn:coerciveness}
        E[\bm f]\geq C\{\norm{\bm f}^p_{W^{1,p}}+\norm{J}_{L^q}^q\}+D,
    \end{align}
    where $C>0$ and $D$ are constants.
    \smallskip
    
    Let $\{\bm f^n\}\subset\mathcal{A}$ be a minimizing sequence for $E[\cdot]$, i.e.
    \begin{align*}
        \lim_{n\rightarrow\infty}E[\bm f^n]=\inf_{\bm f\in\mathcal A}E[\bm f].
    \end{align*}
    From~\eqref{eqn:coerciveness}, we see that $\{\bm f^n\}$ is bounded in $W^{1,p}$, and thus, there exists $\bm f^*\in W^{1,p}(\Omega,\RR^3)$ such that some subsequence $\bm f^k\weakarrow\bm f^*$ weakly in $W^{1,p}$. In view of~\eqref{eqn:local-orientation}, inequality~\eqref{eqn:coerciveness} also implies that $\bm f_{,1}^k\times\bm f_{,2}^k$ is bounded in $L^q(\Omega,\RR^3)$. Hence, there exists $\bm\beta\in L^q(\Omega,\RR^3)$ and a further subsequence (not relabelled) such that $\bm f_{,1}^k\times\bm f_{,2}^k\weakarrow\bm\beta$ weakly in $L^q$.

    For $p\geq2$, the three components of $\bm f_{,1}^k\times\bm f_{,2}^k$ are well defined $L^1$ functions, while for $4/3<p<2$, they should be interpreted in the distributional sense. For example, the first component of $\bm f_{,1}^k\times\bm f_{,2}^k$ along $(1,0,0)$ is understood as
    \begin{align*}
        \int_\Omega \!\!\det\begin{bmatrix}f^k_{2,1}&f^k_{2,2}\\f^k_{3,1}&f^k_{3,2}\end{bmatrix}\phi\dif x:=-\!\!\int_\Omega\begin{bmatrix}f^k_{2,1}&-f^k_{3,1}\\-f^k_{2,2}&f^k_{3,2}\end{bmatrix}\begin{bmatrix}\phi_{,1}\\\phi_{,2}\end{bmatrix}\cdot \begin{bmatrix}f_2^k\\f_3^k\end{bmatrix}\dif x\quad\forall\phi\in C^\infty_c(\Omega).
    \end{align*}
    In any case, it is well known that each of the components converges as a distribution~\cite[Theorem 1.14]{dacorogna2007direct}. Therefore,
    \begin{align*}
        \int_\Omega(\bm f_{,1}^k\times\bm f_{,2}^k)\phi\dif x\rightarrow\int_\Omega(\bm f_{,1}^*\times\bm f_{,2}^*)\phi\dif x\quad\forall\phi\in C^\infty_c(\Omega).
    \end{align*}
    Comparing this with the weak convergence in $L^q$, we conclude that
    \begin{align*}
        \bm f_{,1}^k\times\bm f_{,2}^k\weakarrow\bm f_{,1}^*\times\bm f_{,2}^*\text{ in }L^q.
    \end{align*}
    
    We now focus on the convergence of $J^k:=J[\bm f^k]$. By compact embedding, $\bm f^k\rightarrow\bm f^*$ in $L^p$. Thus, we can extract a further subsequence $\bm f^{k_j}(x)\rightarrow\bm f^*(x)$ that converges pointwise for a.e.~$x\in\Omega$. Since $\bm n$ is a continuous field on $\mathcal{N}$, the dominated convergence theorem implies that $\bm n(\bm f^{k_j})\rightarrow \bm n(\bm f^*)$ strongly in $L^r$ for $1\leq r<\infty$. We may choose $r=q/(q-1)<\infty$ (since $q>1$). Referring again to~\eqref{eqn:local-orientation}, we then deduce
    \begin{align*}
        \int_\Omega J[\bm f^{k_j}]v\dif x = \int_\Omega\bm n(\bm f^{k_j})&\cdot(\bm f_{,1}^{k_j}\times\bm f_{,2}^{k_j})v\dif x\\
        &\rightarrow\int_\Omega \bm n(\bm f^*)\cdot(\bm f_{,1}^*\times\bm f_{,2}^*)v\dif x = \int_\Omega J[\bm f^*]v\dif x,
    \end{align*}
    for all $v\in L^\infty(\Omega)$, i.e.~$J^{k_j}\weakarrow J^*$ weakly in $L^1$. Futhermore, noting that $\bm n(\bm f^*)\in L^\infty(\Omega,\RR^3)$ and $\bm f_{,1}^*\times\bm f_{,2}^*\in L^q(\Omega,\RR^3)$, we deduce that $J^*\in L^q(\Omega)$.
    \smallskip
    
    Next, we show that $\bm f^*\in\mathcal A$. First, $\{\bm f^k\}\subset\mathcal{A}$ implies that $\bm f^*\in W^{1,p}(\Omega,\mathcal{N})$, by the weak closure of the latter. Next, we claim that $J^*>0$ a.e. in $\Omega$. By virtue of Mazur's theorem, we can construct a sequence of convex combinations of the sequence $\{J^k\}$ that converges strongly in $L^1$ to $J^*$. Thus, there is a subsequence converging to $J^*$ a.e.~in $\Omega$. Since each $J^k>0$ a.e., we deduce that $J^*\geq0$ a.e. Now suppose that that $J^*=0$ a.e.~in $\mathcal{O}\subset\Omega$, where $\abs{\mathcal{O}}>0$. Employing $\chi_{\mathcal{O}}$ as a test function, the weak convergence of $J^k$ implies $J^k\rightarrow 0$ strongly in $L^1(\mathcal{O})$. Thus, for a subsequence (not relabelled), $J^k\rightarrow 0$ a.e.~in $\mathcal{O}$. But then~\ref{itm:growth2} and Fatou's lemma imply
    \begin{align*}
        \liminf_{j\rightarrow\infty}E[\bm f^j]\geq\int_{\mathcal{O}}\lim_{j\rightarrow\infty}W(x,\nabla\bm f^j(x))\dif x +C=\infty,
    \end{align*}
    which contradicts our hypothesis that $\inf_{\mathcal{A}}E[\bm f]<\infty$. Hence, $J^*>0$ a.e.~in $\Omega$. Finally, since the Sobolev trace operator $W^{1,p}\rightarrow L^p$ is compact~\cite{necas2011direct}, it follows that $\bm f^{k}\weakarrow\bm f^*$ in $W^{1,p}(\Omega,\RR^3)$ $\implies$ $\bm f^k\rightarrow\bm f^*$ in $L^p(\partial\Omega,\RR^3)$. Thus, $\bm f^*|_{\Gamma}=\bm f_o|_\Gamma$ in the trace sense.
    
    To complete the proof, we combine the results above with Proposition~\ref{prop:wlsc}, below to conclude $E[\bm f^*]\leq\liminf_{k\rightarrow\infty}E[\bm f^k]$ with $\bm f^*\in\mathcal{A}$, i.e., $E$ attains its infimum on $\mathcal{A}$.
\end{proof}

\begin{proposition}\label{prop:wlsc}
    The energy functional~\eqref{eqn:energy} is (sequentially) weakly lower semicontinuous, i.e.,
    \begin{align*}
        \liminf_{k\rightarrow\infty}E[\bm f^k]\geq E[\bm f],
    \end{align*}
    whenever $\bm f^k\weakarrow\bm f$ weakly in $W^{1,p}(\Omega,\mathcal N)$ for $p\geq1$ and $J^k:=J[\bm f^k]\weakarrow J:=J[\bm f]$ weakly in $L^1$, with $J^k,J>0$ a.e. in $\Omega$.
\end{proposition}
\begin{proof}
    Assume (by passing through a subsequence, if necessary) that 
    \begin{align*}
        \lim_{k\rightarrow\infty}E[\bm f^k]=\liminf_{k\rightarrow\infty}E[\bm f^k].
    \end{align*}
    From Mazur's theorem, there exist integers $j(k)\geq k$ and real numbers $c^k_l\geq 0$ for $k\leq l\leq j(k)$ satisfying $\sum_{l=k}^{j(k)}c_l^k=1$ such that
    \begin{align*}
        \sum_{l=k}^{j(k)}c_l^k(\nabla \bm f^l,J^l) \rightarrow (\nabla \bm f,J)\text{ strongly in }L^p(\Omega,\RR^3)\times L^1(\Omega),
    \end{align*}
    as $k\rightarrow\infty$. Therefore, there is a further subsequence
    \begin{align*}
        \sum_{l=k_m}^{j(k_m)}c_l^{k_m}(\nabla \bm f^{l}(x),J^{l}(x)) \rightarrow (\nabla \bm f(x),J(x))\text{ for almost all }x\in\Omega,
    \end{align*}
    where $\lim_{m\rightarrow\infty}k_m=\infty$. Since $(\bm F,J)\mapsto\Phi(x,\bm F,J)$ is convex (therefore continuous) and $J^k,J>0$ for almost all $x\in\Omega$,
    \begin{align*}
        \lim_{m\rightarrow\infty}\Phi\left(x,\sum_{l=k_m}^{j(k_m)}c_l^{k_m}(\nabla \bm f^{l}(x),J^{l}(x))\right)=\Phi(x,\nabla\bm f(x),J(x)),
    \end{align*}
    for almost all $x\in\Omega$.
    
    Then, from Fatou's lemma and convexity,
    \begin{align*}
        E[\bm f] = \int_\Omega \Phi(x,\nabla\bm f(x),J(x))\dif x &\leq \liminf_{m\rightarrow\infty}\int_{\Omega}\Phi\left(x,\sum_{l=k_m}^{j(k_m)}c_l^{k_m}(\nabla \bm f^{l}(x),J^{l}(x))\right)\dif x\\
        &\leq\liminf_{m\rightarrow\infty}\sum_{l=k_m}^{j(k_m)}c_l^{k_m}\int_\Omega \Phi(x,\nabla\bm f^{l}(x),J^{l}(x))\dif x\\
        &=\lim_{m\rightarrow\infty}\int_\Omega\Phi(x,\nabla\bm f^{k_m}(x),J^{k_m}(x))\dif x\\
        &=\lim_{k\rightarrow\infty}\int_\Omega\Phi(x,\nabla\bm f^k(x),J^k(x))\dif x\\
        &=\lim_{k\rightarrow\infty}E[\bm f^k].\qedhere
    \end{align*}
\end{proof}
    
\section{Global Invertibility of Minimizing Configurations}\label{sec:global-invertibility}
We henceforth specialize to a class of Dirichlet-placement problems, viz., we set $\Gamma\equiv\partial\Omega$ in $\mathcal{A}$, cf.~\eqref{eq:admissible-set}. Our goal is to show that an energy minimizer given by Theorem~\ref{thm:existence} is a homeomorphism onto its image. We employ a theorem due to Ball~\cite{ball1981global}, adapted to our setting in $W^{1,p}(\Omega,\mathcal{N})$. We assume $p>2$ throughout this section. We first discuss the Brouwer degree, a key tool generalized to mappings in $W^{1,p}(D,\RR^n)$ in~\cite{ball1981global}, $p>n$, where $D\subset\RR^n$ is a bounded domain.

For any $\bm u\in C(\bar\Omega,\mathcal{N})$, the Brouwer degree of $\bm u$ with respect to $\Omega$ at $\bm y\in\mathcal{N}\setminus\bm u(\partial\Omega)$, denoted $\deg(\bm u,\Omega,\bm y)$, is a well-defined integer depending only on $\bm u|_{\partial\Omega}$, cf.~\cite{milnor1965topology}. In addition, if $\bm u$ is continuously differentiable, we have the formula
\begin{align}\label{eq:degree}
    \deg(\bm u,\Omega,\bm y)=\int_{\Omega}\rho(\bm u(x)-\bm y)J[\bm u(x)]\dif x,
\end{align}
where $\rho:\RR^n\rightarrow[0,\infty)$ is any $C^\infty$ map with compact support in some sufficiently small ball centered at the origin such that $\int_{\RR^n}\rho(\bm x)\dif\bm x=1$, and $J[\bm u]$ is given by the first defining equality ~\eqref{eqn:local-orientation} (without assuming positivity). Equivalently, $J[\bm u]=\sgn\{\bm n(\bm u)\cdot(\bm u_{,1}\times\bm u_{,2})\}\abs{\bm u_{,1}\times\bm u_{,2}}$.  Although formula~\eqref{eq:degree} presumes that $\bm y$ is a regular value, it can be shown that $\bm y\mapsto\deg(\bm u,\Omega,\bm y)$ is constant on connected components of $\mathcal{N}\setminus\bm u(\partial\Omega)$ in the continuous case. For a given $\bm u\in W^{1,p}(\Omega,\mathcal{N})$, there exists a sequence of smooth functions $\bm u_j\rightarrow\bm u$ in $W^{1,p}(\Omega,\mathcal{N})$ and consequently in $C(\bar\Omega,\mathcal{N})$ ($p>2$)~\cite{hajlasz2009sobolev}. As in~\cite{ball1981global}, the right side of~\eqref{eq:degree} defines a continuous functional on $W^{1,p}(\Omega,\mathcal{N})$. Thus, after the substitution of such an approximating sequence, we may take the limit, using the continuity properties of the degree (e.g.,~\cite{kielhofer2011bifurcation}), to deduce that~\eqref{eq:degree} holds for all $\bm u\in W^{1,p}(\Omega,\mathcal{N})$. With this in hand, Theorem 1 (i), (iii) of~\cite{ball1981global} yields:
\begin{proposition}\label{thm:a.e.injectivity}
    Assume that $\bm f_o\in C(\bar\Omega,\mathcal{N})$ is injective in $\Omega$. Then the minimizer $\bm f^*$ of Theorem~\ref{thm:existence} satisfies $\bm f^*(\bar\Omega)=\bm f_o(\bar\Omega)$, and $\bm f^*$ is injective a.e.~in $\Omega$, i.e., $\mathrm{card}(\bm f^{-1}(\bm y))=1$ for almost all $\bm y\in\bm f^*(\Omega)$.
\end{proposition}
\begin{remark}\label{rem:p>2}
    With $p>2$, we may replace~\ref{itm:growth1} (leading to Theorem~\ref{thm:existence}) by: ``There are constants $C_1>0$, $C_2$ such that $W(x,\bm F)\geq C_1\abs{\bm F}^p+C_2$''. For instance, see~\cite{evans2022partial}.
\end{remark}

Next, we replace~\ref{itm:growth1} and strengthen~\ref{itm:growth2} via

\begin{description}
    \item[(H1)$'$\label{itm:growth1'}] There are constants $C_1>0$, $C_2$ and $r>p/(p-2)$ such that
    \begin{align*}
        W(x,\bm F)\geq C_1\left\{\abs{\bm F}^p+J^{-r}\right\}+C_2.
    \end{align*}
\end{description}
Precisely the same arguments employed in Section~\ref{sec:existence} lead to:
\begin{proposition}\label{thm:existence'}
    Given~\ref{itm:growth1'} and~\ref{itm:polyconvexity}, assume the hypotheses of Theorem~\ref{thm:existence} and suppose that $\bm f_o\in\mathcal{A}$ with $E[\bm f_o]<\infty$. Then the total energy~\eqref{eqn:energy} attains its minimum at some $\bm f^*\in\mathcal{A}$, cf.~\eqref{eq:admissible-set}.
\end{proposition}
We specify some terminology before proceeding. Let $\Sigma\subset\mathcal{N}$ be relatively open and bounded. We say that $\Sigma$ has a locally Lipschitz boundary if: For each $\bm y\in\partial\Sigma$, there is an open neighbourhood $\mathcal{O}_{\bm y}\subset\RR^3$ of $\bm y$ and a local coordinate patch for $\mathcal{N}$ such that the inverse image of $\partial\Sigma\cap\mathcal{O}_{\bm y}$ relative to the patch is the graph of a Lipschitz continuous function in the parameter plane. The Sobolev space $W^{1,s}(\Sigma,\RR^2)$ is defined in the standard way via coordinate charts, e.g.,~\cite{taylor1996partial}.  
We now employ Theorem 2 of~\cite{ball1981global} to obtain:
\begin{proposition}\label{thm:homeo}
    Assume the hypotheses of Propositions~\ref{thm:a.e.injectivity} and~\ref{thm:existence'} (with $\Gamma=\partial\Omega$) and suppose that $\bm f_o(\Omega)$ is a locally Lipschitz domain. Then $\bm f^*$ of Proposition~\ref{thm:existence'} is a homeomorphism of $\bar\Omega$ onto $\bm f_o(\bar\Omega)$. Let $\bm h^*$ denote the inverse deformation, viz., $\bm y = \bm f^*(x)\Leftrightarrow x=\bm h^*(\bm y)$. Then we also have $\bm h^*\in W^{1,s}(\bm f_o(\Omega),\RR^2)$, where $s=p(1+r)/(p+r)$.
\end{proposition}
\begin{proof}
    We merely verify the main hypothesis of Theorem 2 of~\cite{ball1981global}, adapted to our setting. For convenience, we drop the superscript $\phantom{}^*$ in what follows. Since $E[\bm f]<\infty$,~\ref{itm:growth1'} implies that $\nabla\bm f\in L^p(\Omega,\RR^3)$ and $(J[\bm f])^{-1}\in L^r(\Omega)$. In order to proceed, we introduce coordinates: Let $\{\bm e_1,\bm e_2\}$ be the standard orthonormal basis for $\RR^2$, and let $\bm y=\tilde{\bm y}(y^1,y^2)\Leftrightarrow y^\alpha=\tilde{y}^\alpha(\bm y)$, $\alpha=1,2$, denote smooth local coordinates on $\mathcal{N}$. Then $\bm y=\bm f(x)\implies y^\alpha =\tilde{y}^\alpha(\bm f(x^1,x^2)):=y^\alpha(x^1,x^2)$. By the chain rule, we then deduce $\nabla\bm f=\diffp{y^\alpha}{{{x^B}}}\bm a_\alpha\otimes\bm e_B$ (summation on $\alpha,B=1,2$), where $\bm a_\alpha := \tilde{\bm y}_{,\alpha}$, $\alpha=1,2$, are the covariant basis vector fields. In other words, $[\nabla\bm f]^\alpha_{\phantom{\alpha}B}=\diffp{y^\alpha}{{{x^B}}}=\bm a^\alpha\cdot(\nabla\bm f\bm e_B)$, where the contravariant basis vector fields are given by $\bm a^\alpha=\nabla\tilde{y}^\alpha$, $\alpha=1,2$. Since $\bm a^\alpha\circ\bm f$ is continuous, each $\diffp{y^\alpha}{{{x^B}}}\in L^p$ on (some open subset of) $\Omega$.

    Let $M$ denote the $2\times 2$ matrix with components $M^\alpha_{\phantom{\alpha}B}:=\diffp{{y^\alpha}}{{{x^B}}}$. In local coordinates, $J=\sqrt{a}\det M$, where $a>0$ is the determinant of the $2\times2$ matrix with components $a_{\mu\nu}=\bm a_{\mu}\cdot\bm a_{\nu}$, cf.~\cite{doyle1956nonlinear}. Since $J>0$ a.e., it follows that $M$ is invertible a.e., the latter designation of which is understood in what follows. By Cramer's rule, $\adj M=(\det M)M^{-1}$, where $\adj M$ denotes the adjugate matrix of $M$. We then have $(\det M)[\nabla\bm f]^{-1}:=[\adj M]^B_{\phantom{B}\alpha}\bm e_B\otimes\bm a^\alpha$, where $[\nabla\bm f]^{-1}$ denotes the inverse of $\nabla\bm f(x)=M^\alpha_{\phantom{\alpha}B}\bm a_\alpha\otimes\bm e_B:\RR^2\rightarrow T_{\bm f(x)}\mathcal{N}$, and $\adj\nabla\bm f:=J[\nabla\bm f]^{-1}=\sqrt{a}[\adj M]^{A}_{\phantom{A}\gamma}\bm e_A\otimes\bm a^\gamma$, or $[\adj\nabla\bm f]^A_{\phantom{A}\gamma}=\sqrt{a}[\adj M]^A_{\phantom{A}\gamma}$. Of course, $\adj M$ and $M$ have the same components (rearranged and to within sign), while $a\circ\bm f$ is continuous. Accordingly, $\abs{\adj\nabla\bm f}\in L^p(\Omega)$ (possibly employing a finite number of coordinate patches combined with a partition-of-unity argument, e.g.,~\cite{munkres2018analysis}).

    Next, H\"older's inequality yields
    \begin{align*}
        \int_\Omega\abs{[\nabla\bm f(x)]^{-1}}^sJ[\bm f(x)]\dif x&=\int_\Omega\abs{\adj\nabla\bm f(x)}^s(J[\bm f(x)])^{1-s}\dif x\\
        &\leq\norm{\adj\nabla\bm f}_{L^p}^s\norm{J[\bm f]^{-1}}_{L^r}^{(s-1)}<\infty.
    \end{align*}
    Since $s>2$, Theorem 2 of~\cite{ball1981global} is valid.
\end{proof}

\begin{remark}
    The componential-Cartesian-based construction in $\RR^n$ employed in the proof of Theorem 2 in~\cite{ball1981global} is precisely the same in our case, but now given in terms of the mixed local ``convected'' coordinates above. However, the covariant derivative must be employed in the identity ensuring that the divergence of $\adj\nabla\bm f$ vanishes, viz., $\left.\left([\adj\nabla\bm f]^A_{\phantom{A}\beta}\right)\right\rvert_A=\left([\adj\nabla\bm f]^A_{\phantom{A}\beta}\right)_{,A}-\left\{\begin{array}{c}\gamma \\ \beta\,\nu\end{array}\right\}\left[\adj\nabla\bm f\right]^A_{\phantom{A}\gamma}\diffp{y^\nu}{{{x^A}}}=0$, where the $\left\{\begin{array}{c}\gamma \\ \beta\,\nu\end{array}\right\}$ denote the Christoffel symbols of the second kind relative to the coordinates on $\mathcal{N}$, cf.~\cite{doyle1956nonlinear}.
\end{remark}

\sloppy
\section{Weak Form of the Spatial Equilibrium Equations}\label{sec:weak-form}
With Proposition~\ref{thm:homeo} in hand, we now generalize results presented in~\cite{ball2002some} to demonstrate that the energy minimizer satisfies a weak spatial (Eulerian) form of the equilibrium equations. We henceforth assume that $W$ is continuously differentiable, and we make an additional, physically reasonable constitutive hypothesis as in~\cite[(C1)]{ball2002some}:

\begin{description}
    \item[(H4)\label{itm:derivative-growth}] There is a constant $K>0$ such that
    \begin{align*}
        \abs{W_{F}(\bm F)\bm F^T}\leq K({W(\bm F)+1}) \text{ for all }\bm F\in\RR^{3\times 2}_+.
    \end{align*}
\end{description}
Hypothesis (C1) of~\cite{ball2002some} pertains to square matrices $\bm F$. We verify hypothesis (H4) for a large class of isotropic membrane energies in Section 6.

Let $\phi:B_\epsilon\times\mathcal{N}\rightarrow\mathcal{N}$, with $B_\epsilon:=(-\epsilon,\epsilon)$, $\epsilon>0$ sufficiently small, be a $C^1$ map such that $\phi(0,\bm y)=\bm y$ and $\phi(\tau,\bm f)|_{\partial\Omega}\equiv\bm f_o$. We assume that $\phi$ and $\Dif\phi$ are uniformly bounded. Let $\Dif_y\phi(\tau,\bm y)\in L(T_y\mathcal N)$ denote the (partial) derivative of $\bm y\mapsto\phi$. We further assume that $\tau\mapsto\Dif_y\phi$ has a continuous derivative on $B_\epsilon\times\mathcal N$, denoted $\Dif_y\dot{\phi}$, which we also take to be uniformly bounded. We now consider spatial variations of the form $\bm f_\tau(x):=\phi(\tau,\bm f(x))$. By the chain rule, we have $\nabla\bm f_\tau(x)=\Dif_y\phi(\tau,\bm f(x))\nabla\bm f(x)$.
\fussy
\begin{proposition}\label{thm:admissibility}
    $\bm f_\tau\in\mathcal{A}$ for all $\tau\in B_\epsilon$.    
\end{proposition}
\begin{proof}
    Observe first that $\Dif_y\phi(0,\bm y)=\bm{1}_y$, the latter denoting the identity on the tangent space $T_y\mathcal{N}$, $\bm y\in\mathcal{N}$. In particular, $\det\bm{1}_y=1$. By assumption, $\Dif_y\phi(\cdot)$ is continuous on $B_\epsilon\times\mathcal{N}$, and $\bm f\in C(\bar\Omega)$ for $p>2$. Thus, $\det\Dif_y\phi(\tau,\bm f(x))>0$ on $B_\epsilon\times\bar{\Omega}$. Also, we observe that for each $\tau\in B_\epsilon$, $\bm f_\tau\in C(\bar\Omega,\RR^3)$ and $\bm f_\tau(x)\in\mathcal{N}$ for all $x\in\bar\Omega$. 

    Employing the local mixed coordinate system introduced in the proof of Proposition~\ref{thm:homeo}, we obtain
    \begin{align*}
        \nabla\bm f_\tau = \Dif_y\phi\nabla\bm f = \diffp{y^\alpha}{{{x^B}}}(\Dif_y\phi\bm a_\alpha)\otimes\bm e_B.
    \end{align*}
    As shown in that proof, each of the component derivatives above belong to $L^p_{loc}(\Omega)$, and thus, $\nabla\bm f_\tau\in L^p(\Omega,\RR^3)$ for each $\tau\in B_\epsilon$. Formally applying the definition $J_\tau:=(\det[\nabla\bm f_\tau^T\nabla\bm f_\tau])^{1/2}$, we deduce $J_\tau=J\det[\Dif_y\phi(\tau,\bm f)]>0$, a.e. in $\Omega$ for all $\tau\in B_\epsilon$, where we have used $J=\sqrt{a}\det\left[\diffp{y^\alpha}{{{x^B}}}\right]>0$ a.e. (locally), cf. Theorem~\ref{thm:existence} and Proposition~\ref{thm:homeo}.
\end{proof}
For $\tau\neq 0$, we now consider the difference quotient
\begin{align}\label{eqn:first-variation-estimate}
    \frac{1}{\tau}(E[\bm f_\tau]-E[\bm f])&=\frac{1}{\tau}\int_\Omega[W(\nabla\bm f_\tau)-W(\nabla\bm f)]\dif x\nonumber\\
    &=\frac{1}{\tau}\int_\Omega\int_0^1\frac{\dif}{\dif s}W(\nabla\bm f_{s\tau})\dif s\dif x\\
    &=\int_\Omega\int_0^1 W_F(\Dif_y\phi(s\tau,\bm f)\nabla\bm f)\cdot\Dif_y\dot{\phi}(s\tau,\bm f)\nabla\bm f\dif s\dif x.\nonumber
\end{align}
In order to proceed, we need a slight generalization of~\cite[Lemma 2.5]{ball2002some}. Let $V\subset\RR^3$ denote an arbitrary 2-dimensional subspace. Assume that $\bm T\in L(V)$ and let $\bm 1$ denote the identity on $V$. Then the same argument in~\cite{ball2002some} yields
\begin{lemma}\label{thm:W-derivative-estimate}
    Given~\ref{itm:derivative-growth}, if $\abs{\bm T-\bm 1}<\delta$, with $\delta>0$ sufficiently small, then there is a constant $C>0$ such that
    \begin{align*}
        \abs{W_F(\bm T\bm A)\bm A^T}\leq C(W(\bm A)+1)\text{ for all }\bm A\in L(\RR^2,V)\cap\RR^{3\times 2}_+.
    \end{align*}
\end{lemma}
\begin{proof}
    In our case, we have $\abs{\bm 1}=\sqrt{2}<2$. As in~\cite{ball2002some}, we may then choose a sufficiently small $\delta$ such that $\abs{\bm T(t)^{-1}}\leq 2$, where $\bm T(t):=t\bm T+(1-t)\bm 1$, for $t\in[0,1]$. Hence, the estimates from~\cite{ball2002some} are the same here, and we arrive at $(1-2K\delta)[W(\bm T\bm A)+1]\leq W(\bm A)+1$. From this and~\ref{itm:derivative-growth} we obtain the desired inequality with $C=2K/(1-2K\delta)$, with $\delta<(2K)^{-1}$.
\end{proof}
By virtue of Lemma~\ref{thm:W-derivative-estimate}, the integrand in the third line of~\eqref{eqn:first-variation-estimate} is bounded by the integrable function $C\abs{W(\nabla\bm f)+1}\sup_{B_\epsilon\times\mathcal{N}}\abs{\Dif_y\dot{\phi}(\tau,\bm y)}$. Thus, we may pass to the limit $\tau\rightarrow0$ in~\eqref{eqn:first-variation-estimate} via the dominated convergence theorem, rigorously obtaining a first-variation condition at the minimizer, viz.,
\begin{align}\label{eqn:first-variation}
    \int_\Omega[W_F(\nabla\bm f)\nabla\bm f^T]\cdot\Dif\bm\psi(\bm f)\dif x = 0,
\end{align}
for all $C^1$ vector fields $\bm\psi$ satisfying $\bm\psi(\bm f)|_{\partial\Omega}=\bm 0$ such that $\bm\psi$ and $\Dif\bm\psi$ are uniformly bounded.

\sloppy
Next, recall the notation introduced in Proposition~\ref{thm:homeo}, viz., $x=\bm h(\bm y):=\bm f^{-1}(\bm y)\Leftrightarrow\bm y=\bm f(x)$; we define $\Omega^*=\bm f(\Omega)$. We also recall the relationship between the first Piola-Kirchhoff stress $\bm S(x):=W_F(\nabla\bm f(x))$ and the Cauchy stress $\bm T(\bm y)=\bm\Sigma(\bm h(\bm y))$, viz., $J(x)\bm\Sigma(x):=\bm S(x)[\nabla\bm f(x)]^T$. Thus, the left side of~\eqref{eqn:first-variation} is equivalent to $\int_\Omega\bm\Sigma(x)\cdot\Dif\bm\psi(\bm f(x))J(x)\dif x,$ and the change of variables formula as in~\cite{ball1981global} yields
\begin{proposition}\label{eqn:equilibrium-equation}
    Assume the hypotheses of Proposition~\ref{thm:homeo} along with~\ref{itm:derivative-growth}. Then the weak spatial form of the equilibrium equations holds at an energy minimizer $\bm y=\bm f(x)$, viz.,
    \begin{align*}
        \int_{\Omega^*}\bm T(\bm y)\cdot\Dif\bm\psi(\bm y)\dif\bm y=0,
    \end{align*}
    for all vector fields $\bm\psi$ as specified after~\eqref{eqn:first-variation} and satisfying $\bm\psi|_{\partial\Omega^*}=\bm 0$, where $\partial\Omega^*:=\bm f(\partial\Omega)$.
\end{proposition}
\fussy

\section{Concluding Remarks}\label{sec:concluding-remarks}
 A finite-strain membrane model capable of sustaining compression is a crucial ingredient in the accurate prediction of wrinkling phenomena in highly stretched thin elastic sheets, e.g., ~\cite{healey2023existence,li2016stability}. Finite strains are not captured by thin-plate models derived rigorously via $\Gamma$-convergence~\cite{muller2017mathematical}. Moreover, the same approach (via different scaling) to pure membrane models leads to tension-field theory, as already discussed in Section~\ref{sec:intro}. Taken together, this provides the primary motivation for direct modeling such as that proposed in this work. Here the unit normal of the prescribed oriented surface, on which the membrane deforms, plays a similar role as the director field in finite-deformation Cosserat shell theory~\cite{healey2023energy}. Weak convergence of the oriented local area ratio can be obtained, and minimum-energy existence theorems are established in both cases.

Hypotheses~\ref{itm:polyconvexity}-\ref{itm:derivative-growth} can be verified for a class of isotropic materials. We first address~\ref{itm:polyconvexity} and \ref{itm:growth2} for energy densities of the form $W(\bm F)=\Psi(\bm F)+\Theta(J)$. Suppose that $\Psi:\RR^{3\times 2}\rightarrow[0,\infty)$ and $\Theta:(0,\infty)\rightarrow[0,+\infty)$ are each convex, with $\Theta(J)\nearrow\infty$ as $J\searrow0$, where $J=\abs{\bm f_{,1}\times\bm f_{,2}}$. Clearly the two hypotheses are satisfied. Presuming material objectivity~\eqref{eq:material-objectivity}, the material is isotropic if $W(\bm F\bm R)\equiv W(\bm F)$ for all $\bm R\in SO(2)$. This implies $W(\bm F)=\Phi(\lambda_1,\lambda_2)=\Phi(\lambda_2,\lambda_1)$, where $\lambda_1^2,\lambda_2^2>0$ denote the eigenvalues of the right Cauchy-Green Strain tensor $\bm C=\bm F^T\bm F\in\RR^{2\times2}$ and $\lambda_1,\lambda_2>0$ are the principal stretches. Let $a_1,a_2\geq0$ denote the singular values $\bm A\in\RR^{3\times2}$. Likewise, isotropy implies $\Psi(\bm A)=\Upsilon(a_1,a_2)=\Upsilon(a_2,a_1)$. By virtue of~\cite[Theorem 5.34B \& Remark 5.36]{dacorogna2007direct}, we deduce that the convexity of $\Psi$ is the same as the convexity of $\Upsilon$ and $a_\alpha\mapsto\Upsilon$ non-decreasing, $\alpha=1,2$. Hence, isotropic stored energies of the form
\begin{align*}
    W(\bm F)=\Upsilon(\lambda_1,\lambda_2)+\Theta(\lambda_1\lambda_2),
\end{align*}
with $\Upsilon$ having the aforementioned properties, satisfy hypotheses~\ref{itm:polyconvexity} and \ref{itm:growth2}. A family of examples is afforded by  $\Upsilon(\lambda_1,\lambda_2)=\sum_{j=1}^Mb_j(\lambda_1^{\gamma_j}+\lambda_2^{\gamma_j})$, $b_j>0$, $\gamma_j\geq 1$, $j=1,...,M$. We give another example below. Incidentally, hypothesis~\ref{itm:growth1} holds for these materials provided that $\max_j\gamma_j>4/3$ and $\Theta(J)\geq C J^q$; the latter are not required if $\max_j\gamma_j\geq2$, cf. Remark~\ref{rem:p>2}.

We now turn to hypothesis~\ref{itm:derivative-growth}, the left side of which represents the magnitude of what is sometimes called the Kirchhoff stress, denoted $\bm\sigma:=W_{\bm F}\bm F^T$. Let $\mathcal{V}\subset\RR^{3}$ denote the two-dimensional range of $\bm F\in\RR^{3\times 2}$. Without changing notation, we consider the invertible linear transformation $\bm F:\RR^2\rightarrow\mathcal{V}$. Note that $W_{\bm F}(\bm F)\in L(\RR^2,\mathcal{V})$, $\bm F^T\in L(\mathcal{V},\RR^2)$, and thus, $\bm\sigma\in L(\mathcal{V})$. For isotropic materials, it can be shown that $\bm\sigma=\sum_{\gamma=1}^2\lambda_\gamma\Phi_{,\gamma}\bm d_\gamma\otimes\bm d_\gamma$, where $\bm d_1,\bm d_2$ are orthonormal eigenvectors of $\bm B=\bm F\bm F^T\in L(\mathcal{V})$, e.g.,~\cite[(4.3.42)]{ogden1984non}. Therefore,~\ref{itm:derivative-growth} reads
\begin{align}
   \label{itm:derivative-growth-iso}\abs{(\lambda_1\Phi_{,1},\lambda_2\Phi_{,2})}\leq K(\Phi+1).\tag*{(H4)\textsubscript{iso}}
\end{align}
This is the two-dimensional version of the inequality given in~\cite{ball2002some} regarding bulk three-dimensional isotropic materials.

Hypothesis~\ref{itm:derivative-growth-iso} is satisfied for the class of two-dimensional isotropic materials presented above. This is similar to the claim in~\cite{ball2002some} regarding Ogden materials. For a different example, consider
\begin{align}\label{eq:isotropic-energy-example}
    \Phi(\lambda_1,\lambda_2)=a(\lambda_1^\gamma+\lambda_2^\gamma)+b[(\lambda_1)^2+(\lambda_2)^2](\lambda_1\lambda_2)^{-1}+\Theta(\lambda_1\lambda_2),
\end{align}
where $a,b>0$, $\gamma\geq1$, and $\Theta$ is as described above. A simple computation yields
\begin{align*}
    \lambda_1\Phi_{,1}=a\gamma\lambda_1^\gamma+b[(\lambda_1)^2-(\lambda_2)^2](\lambda_1\lambda_2)^{-1}+\lambda_1\lambda_2\Theta'(\lambda_1\lambda_2),
\end{align*}
and a similar expression follows for $\lambda_2\Phi_{,2}$ (by swapping indices). Clearly, we can find a constant $K>0$ so that~\ref{itm:derivative-growth-iso} is satisfied, provided that $\abs{J\Theta'(J)}\leq C(\Theta(J)+1)$ for all $J>0$. The latter condition is also given in~\cite{ball2002some}.

We mention that~\eqref{eq:isotropic-energy-example} satisfies hypotheses~\ref{itm:polyconvexity} and~\ref{itm:growth2}. Regarding the former, this follows by expressing the second term on the right side of~\eqref{eq:isotropic-energy-example} as ``$b(\bm F\cdot\bm F)J^{-1}$'', which is readily shown to be convex in $(\bm F,J)$ via differentiation. The latter is given in~\cite{rosakis1994relation} as an example of a polyconvex energy density for bulk two-dimensional nonlinear elasticity. We also take the opportunity to note that the function $(\bm F,J)\mapsto(\bm F\cdot\bm F)J^{-2}$ is not convex. If this term replaces the second term in~\eqref{eq:isotropic-energy-example}, the resulting energy density apparently does not satisfy~\ref{itm:polyconvexity}, nor is it polyconvex in the context of bulk two-dimensional nonlinear elasticity. In any case, this corrects a statement made in~\cite{healey2023existence} in reference to a similar example (with $\gamma = 2$) presented in~\cite{muller2004rubber}.

As in~\cite{ball2002some}, the weak form of the Eulerian equilibrium equations given in Proposition 11 does not imply the weak form of the Euler-Lagrange equations of~\eqref{eqn:energy}. Another result of~\cite{ball2002some}, based on an assumption akin to~\ref{itm:derivative-growth}, rigorously ensures the weak form of the conservation law for the Eshelby energy-momentum tensor. This carries over directly to our case as well, since the so-called inner variation is taken in the flat reference configuration $\Omega\subset\RR^2$.

The generalization of our problem to nonlinearly elastic maps in a 3-dimensional setting is clear. This corresponds to the deformation of a bulk hyperelastic body into a nontrivial (nice enough) 3-dimensional manifold without boundary. However, a physical realization of the problem is not apparent.

\appendix
\section*{Appendix}
We provide a proof of the fact that stored energies $W:\bar\Omega\times\RR^{3\times2}_+\rightarrow[0,\infty)$ satisfying material objectivity ~\eqref{eq:material-objectivity} and hypothesis ~\ref{itm:growth2} are not rank-one convex. We make no claim of originality;  the construction was shown to us anonymously. The result seems to be well known (at least believed) but not written down precisely.
Accordingly, we provide it here: Define
\begin{align*}
    \bm F^\pm := \begin{bmatrix}
        \lambda & 0\\
        0 & \epsilon\mu\\
        0 & \pm\mu\sqrt{1-\epsilon^2}
    \end{bmatrix},
\end{align*}
where $\lambda,\mu>0$, and $\epsilon>0$ is sufficiently small. Then
\begin{align*}
    \bm F^+-\bm F^- = \begin{bmatrix}
        0 & 0\\
        0 & 0\\
        0 & 2\mu\sqrt{1-\epsilon^2}
    \end{bmatrix}\text{ and }\overline{\bm F}:=\frac{1}{2}(\bm F^++\bm F^-)=\begin{bmatrix}
        \lambda & 0\\
        0 & \epsilon\mu\\
        0 & 0
    \end{bmatrix}.
\end{align*}
Note that $\bm F^+$ and $\bm F^-$ are rank-one connected. Objectivity~\eqref{eq:material-objectivity} implies $W(\bm F)=\Psi(\bm C)$, where $\Psi$ is a function of the right Cauchy-Green tensor $\bm C=\bm F^T\bm F$. Simple computations yield 
\begin{align*}
    \tilde{\bm C}:=\bm C^+=\bm C^-=
    \begin{bmatrix}
        \lambda^2 & 0\\
        0 & \mu^2
    \end{bmatrix},\text{ and }\overline{\bm C}=\begin{bmatrix}
        \lambda^2 & 0\\
        0 & \epsilon^2\mu^2
    \end{bmatrix}, 
\end{align*}
where $\bm C^\pm$ correspond to $\bm F^\pm$, respectively, and $\overline{\bm C}$ to $\overline{\bm F}$. Thus, $\Psi(\tilde{\bm C})=W(\bm F^+)=W(\bm F^-)$ and $\Psi(\overline{\bm C})=W(\overline{\bm F})$. Clearly, $\det{\tilde{\bm C}}=(\lambda\mu)^2$ and $\det\bar{\bm C}=(\epsilon\lambda\mu)^2$. From~\ref{itm:growth2} we have $W\rightarrow +\infty$ as $J\rightarrow0^+,$ where $J^2=\det\bm C$. Thus, $\Psi(\bar{\bm C})>\Psi(\tilde{\bm C})$ for sufficiently small $\epsilon$. But this is the same as
\begin{align*}
    W\left(\frac{1}{2}[\bm F^+ +\bm F^-]\right)>\frac{1}{2}[W(\bm F^+)+W(\bm F^-)],
\end{align*}
where $\bm F^+$ and $\bm F^-$ are rank-one connected.\qed

\section*{Acknowledgments}
This work was supported in part by the National Science Foundation through grant DMS-2006586, which is gratefully acknowledged. GGN thanks the Hausdorff Research Institute for Mathematics in Bonn for its hospitality during the Trimester Program, \textit{Mathematics for Complex Materials} funded by the Deutsche Forschungsgemeinschaft (DFG, German Research Foundation) under Germany's Excellence Strategy -- EXC-2047/1 -- 390685813.

\bibliographystyle{spmsci.bst}
\bibliography{ref}
\end{document}